\documentclass[11pt,twoside]{article}
\usepackage{latexsym}
\usepackage{amssymb,amsbsy,amsmath,amsthm,amsfonts,amssymb,amscd,stmaryrd,mathtools}
\usepackage{epsfig, graphicx,mathrsfs}
\usepackage{esint}
\usepackage[utf8]{inputenc}
\usepackage[T1]{fontenc}
\usepackage{dsfont}
\usepackage{empheq}
\usepackage{color}
\newcommand{\commentout}[1]{}
\newcommand{\R}{\mathbb{R}}
\newcommand{\N}{\mathbb{N}}

\newcommand {\ep}  {\varepsilon}
\newcommand {\ffi}  {\varphi}

\newcommand {\Chi} {{\bf \raise 2pt \hbox{$\chi$}} }
\newcommand{\dd}{\mathrm{d}}

\renewcommand{\H}{\textnormal{H}}

\newcommand{\Id}{\textnormal{Id}}
\newcommand{\beq}{\begin{equation}}
\newcommand{\eeq}{\end{equation}}
\newcommand{\bea} {\begin{array}{rl}}
\newcommand{\eea} {\end{array}}
\newcommand{\bepa}{\left\{ \begin{array}{l}}
\newcommand{\eepa} {\end{array}\right.}

\newcommand{\din}{\dot{\,\in\,}}
\newcommand{\ddin}{\ddot{\,\in\,}}

\def\T{\mathbb{T}}
\def\N{\mathbb{N}}
\def\R{\mathbb{R}}

\def\H{\textnormal{H}}

\def\Id{\textnormal{Id}}

\newtheorem{lem}{Lemma}
\newtheorem{thm}{Theorem}

\newtheorem{Propo}{Proposition}
\newtheorem{definition}{Definition}
\newtheorem{remark}{Remark}
\newtheorem{question}{Question}

\title{From non-local to classical SKT systems :\\ triangular case with bounded coefficients}
                       \author{Ayman Moussa\footnote{Sorbonne-Universit\'e, Universit\'e Paris-Diderot SPC, CNRS, UMR 7598 Laboratoire Jacques-Louis Lions, F-75005 Paris, France}}
\begin{document}
\maketitle
\abstract{This paper solves partially a question suggested by Fontbona and Méléard. The issue is to obtain rigorously cross-diffusion systems \emph{à la} Shigesada-Kawasaki-Teramoto as the limit of relaxed systems in which the cross-diffusion and reaction coefficients are non-local. We depart from the existence result established by Fontbona and Méléard for a general class of non-local systems and study the corresponding asymptotic as the convolution kernels tend to Dirac masses, but only in the case of (strictly) triangular systems, with bounded coefficients. Our approach is based on a new result of compactness for the Kolmogorov equation, which is reminiscent of the celebrated duality lemma of Michel Pierre.}
\section{Introduction}
In 1979, Shigesada, Kawasaki and Teramoto introduced in \cite{Shigesada1979} the following system (that we denote SKT), on $Q_T:=[0,T]\times\Omega$ where $\Omega\subset\R^N$ is some regular bounded open set
\begin{align}
\label{SKT}\left\{
\begin{array}{l}
\partial_t u_1-\Delta \Big[(d_1+a_{11}u_1+a_{12}u_2)u_1\Big]= u_1(\rho_1-s_{11}u_1-s_{12}u_2),\\
\\
\partial_t u_2-\Delta \Big[(d_2+a_{22}u_2+a_{21}u_1)u_2\Big]= u_2(\rho_2-s_{21}u_1-s_{12}u_2).
\end{array}
\right.
\end{align}
The latter aims at describing the behavior of two populations (through their density functions $u_1,u_2\geq 0$) involving different mechanisms: self-diffusion ($a_{11}$, $a_{22}$ terms), cross-diffusion ($a_{12},a_{21}$ terms) and growth terms modelling reproduction ($\rho_1,\rho_2$ terms) or competition ($s_{ij}$ terms). The SKT system was first introduced because of the interesting properties of its steady states (it allows the formation of segregation patterns), see for instance \cite{MimYam,MimKaw,Kishimoto1985,Ni_et_al}. On the other hand the existence theory for the SKT system and its variants (multi-species, nonlinear diffusion or reaction coefficients \emph{etc}) has been puzzling the mathematical community for three decades now. To summarize, there exists two main strategies to establish existence results for these systems. The first strategy produces classical solutions using Amman's theorem (see \cite{amann88,amann90}) and leads  \emph{a priori} to local solution unless certain Sobolev norms are controlled. This first method allows the construction of global solutions for the SKT system under strong restrictions on the coefficients (see \cite{toan1,toan2,LouNiWU1998} or more recently \cite{HoanNguPha,lou2015global}). The second strategy uses the entropy structure of the SKT system, exhibited for the first time (in the 1D case) by Galiano \emph{et. al.} in \cite{Galiano_Num_Math} and then used to solve globally the SKT system in any dimension by Chen and Jüngel in \cite{Chen2004,Chen2006}. These last references can be seen as the end of the long road to the existence of global weak solutions for the SKT system, which was punctuated by several partial results \emph{e.g.} \cite{Yagi,Wang2005}. At the same time the discovery of this Lyapunov functionnal was a turning point: the entropy structure appeared to be robust enough to allow the elaboration of different approximation schemes for numerous variants of the SKT system (see \cite{chen2016global,dlmt,lepmou} and the reference therein).
 Let us also mention another tool, coming from reaction-diffusion theory, intensively used in some of the existing proofs of global existence. The duality lemma is an \emph{a priori} estimate inspired by the papers \cite{MaPi,PiSc} (see \cite{dualimpro} for improved versions), which allows in one go to justifiy the integrability of each of the nonlinearities of the system \eqref{SKT} (and is thus useful to handle concentration issues in the approximation process); together with the entropy estimate (allowing for gradient estimates and thus strong compactness) they form the cornerstone of the global existence results obtained in \cite{DesLepMou,dlmt,lepmou}. Let us finally mention a recent result of Chen and Jüngel concerning the uniqueness of weak solutions: \cite{unique}.

\vspace{2mm}

Apart from the two aforementioned approaches for the existence theory (regular solutions \emph{via} Amman's theorem, or weak solutions \emph{via} entropy method), another possibility (far less used in the literature) to prove existence of (weak) solutions to SKT systems is to realize these solutions as limit of simpler systems. For instance, it is possible to recover (see \cite{murakawa,condes,Mimura_cross,MimIz,ariane,DesTres} and the recent improvements \cite{estlaujun,dausdesdie}) the SKT system as the singular limit of a reaction-diffusion coupling in which one of the two populations exists in two states and passes from one to the other depending on the density of the other species. Not only this method produces a solution but it gives also some insight on the meaning of the cross-diffusion phenomenon, and can be seen as a way to justify this model. Another example using this asymptotic approach but yet from a completely different point of view is the article \cite{melfont} of Fontbona and Méléard. They managed to prove existence of a non-local version of the SKT model departing from a (finite) population process and studying its limit when the number of individuals tends to infinity. An interesting aspect of this work is that the tools involved are of a complete different nature when compared to all the literature that we mentionned previously: mostly probabilistic techniques. As a result of this first asymptotic, they obtain a global weak non-negative solution for the following SKT-type system, for $i\in\{1,\dots,M\}$
\begin{multline}
\label{syst:nonloc}\partial_t u_i = \frac{1}{2} \sum_{k,l=1}^d \partial_{x_kx_l}^2 (a_{k,l}^i(\cdot,G^{i1}\star u^1,\cdots,G^{iM}\star u^M)u^i) \\
- \sum_{k=1}^d \partial_{x_k} (b_{k}^i(\cdot,H^{i1}\star u^1,\cdots,H^{iM}\star u^M)u^i) + \left(r_i-\sum_{j=1}^M C^{ij}\star u^j\right)u^i,
\end{multline}
where the convolution $\star$ acts only on the space variable and the functions $a_{k,l}^i, b_k^i, G^{ij}, H^{ij}, C^{ij}$ are given regular functions, the unknown being the vector $(u_1,\dots,u_M)$. After establishing the uniqueness for a class of solutions of system \eqref{syst:nonloc}, Fontbona and Méléard study the limit $C^{ij}\rightharpoonup c_{ij} \delta$ where $c_{ij}\in\R$ and $\delta$ is the Dirac mass and prove (under adequate assumptions on the functions $a_{k,l}^i, b_k^i, G^{ij}, H^{ij}, C^{ij}$) that the corresponding solutions of \eqref{syst:nonloc} do converge to a solution of the following system for $i\in\{1,\dots,M\}$
\begin{multline}
\label{syst:nonlocbis}\partial_t u_i = \frac{1}{2} \sum_{k,l=1}^d \partial_{x_kx_l}^2 (a_{k,l}^i(\cdot,G^{i1}\star u^1,\cdots,G^{iM}\star u^M)u^i) \\
- \sum_{k=1}^d \partial_{x_k} (b_{k}^i(\cdot,H^{i1}\star u^1,\cdots,H^{iM}\star u^M)u^i) + \left(r_i-\sum_{j=1}^M c_{ij} u^j\right)u^i.
\end{multline}
Note that the SKT system can be seen as a particular case of the first system \eqref{syst:nonloc}, for which $b_k^i =0$, $a_{k,l}^i=\delta_{k,l}$ where $a^i$ is an affine function (independent of $x$) and each convolution kernel $G^{ij},H^{ij}, C^{ij}$ is replaced by a Dirac mass. Systems like \eqref{syst:nonloc} or \eqref{syst:nonlocbis} are often called \emph{non-local} or \emph{relaxed} cross-diffusion systems. The first naming (non-local) is obvious: some ponctual interactions of the SKT system are replaced by long-range ones, allowing the individuals of different species to compete or diffuse one another from far away, the distance of mutual interaction being measured by the diameter of the support of the different convolution kernels. The second (relaxed) naming refers to the mathematical stiffness of the system: it makes no doubt that the nonlinearities of system \eqref{syst:nonloc} or \eqref{syst:nonlocbis} are a lot more handable than the ones of the original SKT system \eqref{SKT}, because the convolution kernels are regular functions that smooth out the behaviour of the unknown. However the presence of these mollifying operators destroys completely the entropy structure discussed above for the SKT system and its variants:  up to now no generic Lyapunov functionnals have been exhibited for relaxed cross-diffusion systems. Another type of non-local cross-diffusion system was introduced in \cite{Lepoutre_JMPA} and investigated in \cite{LPR}: the convolution operator is replaced by $(\Id-\ep_i \Delta)^{-1}$ where $\ep_i>0$ is a small parameter. This elliptic regularization shares some properties with the convolution and is easier to manipulate when considering a boundary-value problem (which is the case in \cite{LPR}). No reaction terms are considered in \cite{LPR}, but the authors manage to prove existence and uniqueness of global and regular non-negative solutions for a wide family of relaxed cross-diffusion system. As it is the case for the convolution relaxation, the elliptic one forbids the use of the entropy machinery described earlier. For this reason, the mathematical analysis performed in \cite{LPR} is different from the known literature on global solutions for the SKT models and uses tools of a quite different nature that are generic parabolic estimates which are usually not available when considering the classical SKT systems because \emph{the latter cannot be written as parabolic equations with regular (or even) bounded coefficients}. The authors of \cite{LPR} makes also a clever use of the aforementionned duality estimate in the construction of their solutions. As precised earlier, no reaction terms (either local or not) are included in \cite{LPR}, but we can at the opposite mention several reaction-diffusion papers including non-local right hand sides (but no cross-diffusion terms) as for example \cite{coville,bernaper,genieys} and the references therein.

\vspace{2mm}

The results obtained in \cite{melfont} (that are: existence of solutions to \eqref{syst:nonloc} \emph{via} probabilistic methods then rigorous convergence of these solutions towards solutions of \eqref{syst:nonlocbis} as the reaction mollifications vanish) can be seen as the first step toward the full rigorous derivation of SKT-like models by an individual-based approach. Unfortunately Fontbona and Méléard did not manage to pass to the limit in the remaining convolution operators, that is studying the asymptotic $(G^{ij},H^{ij})\rightharpoonup (\delta,\delta)$ which, to quote them, appear to be a ``highly difficult open challenge''. The present paper is located right after this analysis and establishes the limit from this non-local system to the classical one but only in the case of \emph{triangular} SKT systems, without \emph{self-diffusion} assuming that the coefficients inside the diffusion operator are \emph{bounded}. More precisely we look at SKT relaxed systems having the following form
\begin{empheq}[left = \textnormal{(SKTR)}\quad\empheqlbrace]{align*}
&\partial_t u_1 - \Delta[a_1(\,\cdot\,,u_2\star\rho_2,u_3\star\rho_3,\dots,u_I\star\rho_I)u_1] = r_1(u_1\star \rho_1,\dots,u_M\star\rho_I)u_1,\\
&\partial_t u_2 - \Delta[a_2(\,\cdot\,,u_3\star\rho_3,\dots,u_I\star\rho_I)u_2] = r_2(u_1\star\rho_1,\dots,u_I\star\rho_I)u_2, \\
&\hspace{2mm}\vdots\\
&\partial_t u_I - \Delta[a_I u_I] = r_I(u_1\star\rho_1,\dots,u_I\star\rho_I)u_I, 
\end{empheq}
where $a_i$ and $r_i$ are continuous functions the first one being bounded from above and below by positive constants, and the second one only from above. Such a system is called (strictly) triangular because the diffusion coefficient for the $i$-th population depends upon the next $M-i$ other species only; in particular this coefficient does not depend on the population itself: no self-diffusion. We work on the periodic torus and do not pay to much assumption here to the question of global existence for (SKTR) (even though we include a result as a by-product of our analysis): we use the solutions built by Fontbona and Méléard on the whole space (the construction being similar on the torus). We recall that one could obtain likewise bounded solutions \emph{via} a PDE approach, following the lines of \cite{LPR} (a little effort is needed to include the reaction terms). Our main interest here is to study rigorously the limit $\rho_i \rightharpoonup \delta$, proving that the weak solutions of (SKTR) converge to a weak solution of the following one: 
\begin{empheq}[left = \empheqlbrace]{align*}
&\partial_t u_1 - \Delta[a_1(\,\cdot\,,u_2,u_3,\dots,u_I)u_1] = r_1(u_1,\dots,u_I)u_1,\\
&\partial_t u_2 - \Delta[a_2(\,\cdot\,,u_3,\dots,u_I)u_2] = r_2(u_1,\dots,u_I)u_2, \\
&\hspace{2mm}\vdots\\
&\partial_t u_I - \Delta[a_I u_I] = r_I(u_1,\dots,u_I)u_I.
\end{empheq}
Without uniqueness for the limit system, we only recover the convergence of a subsequence, by a compactness argument. At first sight, it is not obvious to exhibit (strong) compactness for a sequence of solutions to SKTR : no gradient estimates are available for such non-local cross-diffusion systems. We obtain this compactness property by a careful study of the \emph{Kolmogorov equation}, that is an equation of the form 
\begin{align} 
\label{intro:eq:kol1}
\partial_t z -\Delta(\mu z) &= G,\\
\label{intro:eq:kol2}z(0,\cdot) &=z^0,
\end{align}
where $\mu$, $G$ and $z^0$ are given and $z$ is the unknown. Each equations of the cross-diffusion systems above is a particular instance of the Kolmogorov equation. We consider a rather weak functional framework to solve \eqref{intro:eq:kol1} -- \eqref{intro:eq:kol2} in which no strong regularity assumptions is needed on $\mu$ : it is only assumed to be bounded from above and below (away from $0$). This low regularity forbids to use standard parabolic estimates for equations \eqref{intro:eq:kol1} ; it is consistent with the study of cross-diffusion systems like above since very few results of regularity are known for these systems, especially when one considers functions $a_i$ and $r_i$ which are merely continuous and an unbounded initial data, as we do.

\medskip

We propose in a dedicated section a self-contained exploration of the problem \eqref{intro:eq:kol1} -- \eqref{intro:eq:kol2}, when $\mu$ is bounded from above and below. We prove that under this assumption, the problem \eqref{intro:eq:kol1} -- \eqref{intro:eq:kol2} is well-posed, satisfies a weak maximum principle and also a strong stability property. This analysis is performed by using the \emph{dual problem} of \eqref{intro:eq:kol1} -- \eqref{intro:eq:kol2}, following the pioneer work of Martin, Pierre and Schmitt. We identify incidentally a non-trivial question : on which condition on $\mu$, uniqueness holds for the dual problem ? We give a list of sufficient conditions, the boundedness from above and below being one of them. In fact, it is precisely the well-posedness of the dual problem which allows to recover most of the useful properties that we exhibit on the Kolmogorov equation. Among these properties the aforementioned strong stability transfers $\textnormal{L}^1(Q_T)$ convergence on $\mu$ to $\textnormal{L}^2(Q_T)$ on $z$, under mere bounds for the source term $G$ and initial data $z^0$. Exploiting the uniqueness for the dual equation to prove a stability result on the Kolmogorov equation is one of the main innovation of this work.

\medskip

Once this strong type stability estimate is established for \eqref{intro:eq:kol1}, we translate it into a compactness result for a more general source term of the form $G=Rz$ (we failed to prove uniqueness for the corresponding modified Kolmogorov equation). Our strategy of proof to tackle the asymptotic limit $\rho_i\rightharpoonup \delta$ on the system (SKTR) is then to propagate compactness from one equation to the other. It is because of this trigger effect that our result applies to triangular systems (only) : we need strong compactness of \emph{one} of the populations to contaminate the other ones, one by one, the onset of this mechanism being given by the last equation (which is the less coupled in some sense). 

\medskip

Let's expose the structure of this work. In Section \ref{sec:state} we introduce some notations and state our main results. Section \ref{sec:dualsol} proposes a thorough study of the Kolmogorov equation, under a low regularity assumption for the diffusion coefficient ; this is done by studying the dual equation of \eqref{intro:eq:kol1}, that we introduce therein. The main result of this section is a stability result thanks to which we are able to prove our central compactness lemma. The proof of the latter is given at the beginning of Section~\ref{sec:proofth}. At this stage, only single equations were considered and the remaining part of Section~\ref{sec:proofth} aims at using these scalar results on systems. First, in Subsection \ref{subsec:conv} we prove the convergence of non-local triangular cross-diffusion systems to a classical one, when the kernels tend to Dirac measures. Then, in Subsection \ref{subsec:exi}, we prove an existence result for a large class of non-local triangular systems with continuous coefficients, to which our main result of convergence can be applied. In Section~\ref{sec:com} we discuss the uniqueness issue for the dual problem and explains how one can considerably enhance our main results by proving uniqueness for a larger class of diffusion coefficients. 
Finally, in the Appendix Section \ref{sec:app} we collect several technical results.
\section{Notations and main results}\label{sec:state}
\subsection{Notations}\label{sec:not}
In all this article $T>0$ and $N\in\N$ are fixed and $Q_T$ denotes the periodic parabolic cylinder $[0,T[\times \T^N$ . We denote by $\mathscr{D}(Q_T)$ the vector space of all smooth functions defined on $Q_T$ having a compact support and by $\mathscr{D}'(Q_T)$ its dual space.

\medskip

\noindent The norm of any normed vector space $X$ is denoted $\|\cdot\|_X$.

\medskip

\noindent For any sequence $(x_n)_n$ of a normed vector space $X$, $(x_n)_n\din X$ means ``$(x_n)_n$ is bounded in $x$'' and $(x_n)_n\ddin X$ means ``$(x_n)_n$ admits a converging subsequence in $X$''. Without more precisions, $(x_n)_n\ddin X$ refers to the topology induced by the norm of $X$.
\subsection{Statements}
Here is our main result
\begin{thm}\label{thm:nonloctoclas}
Consider two families of real-valued functions $(a_i)_{1\leq i\leq I}$ and $(r_i)_{1\leq i\leq I}$, where for all $i$ the function $a_i$ is defined on $\overline{Q_T}\times \R^{I-i}$ and $r_i$ is defined on $\R^I$. For all $n\geq 1$ fix a family of non-negative (and normalized) smooth kernels $(\rho_i^n)_{1\leq i\leq I}$ and consider a family of  non-negative functions $(u_{i}^n)_{1\leq i \leq I}\in\textnormal{L}^2(Q_T)$ solution of the following Cauchy problem on $Q_T$
\begin{empheq}[left = \empheqlbrace]{align*}
&\partial_t u_1^n - \Delta[a_1(\,\cdot\,,u_2^n\star\rho_2^n,u_3^n\star\rho_3^n,\dots,u_I^n\star\rho_I^n)u_1^n] = r_1(u_1^n\star\rho_1^n,\dots,u_I^n\star\rho_I^n)u_1^n,\\
&\partial_t u_2^n - \Delta[a_2(\,\cdot\,,u_3^n\star\rho_3^n,\dots,u_I^n\star\rho_I^n)u_2^n] = r_2(u_1^n\star\rho_1^n,\dots,u_I^n\star\rho_I^n)u_2^n, \\
&\hspace{2mm}\vdots\\
&\partial_t u_I^n - \Delta[a_I u_I^n] = r_I(u_1^n\star\rho_1^n,\dots,u_I^n\star \rho_I^n)u_I^n, 
\end{empheq}
with initial condition $u_i^n(0,\cdot) = u_i^{in} \in\textnormal{L}^2(\T^N)$, for all $i\in\llbracket 1,I\rrbracket$. Assume that the functions $a_i, r_i$ are continuous, with $a_i$ bounded from above and below by a positive constant and $r_i$ only from above. Assume furthermore that  $|r_i|$ is sub-affine, that is $|r_i(x_1,\cdots,x_I)|\lesssim 1+|x_1|+\cdots+|x_I|$. 

\vspace{2mm} 

If $(\rho_i^n)_n\rightharpoonup_n \delta$  (the Dirac mass) in $\mathscr{D}'(\T^N)$ for all $i$, then up to a subsequence, $(u_i^n)_{1\leq i\leq I}$ converges strongly in $\textnormal{L}^2(Q_T)$ to a solution $(u_i)_{1\leq i\leq I}$ of the following cross-diffusion system
\begin{empheq}[left = \empheqlbrace]{align*}
&\partial_t u_1 - \Delta[a_1(\,\cdot\,,u_2,u_3,\dots,u_I)u_1] = r_1(u_1,\dots,u_I)u_1,\\
&\partial_t u_2 - \Delta[a_2(\,\cdot\,,u_3,\dots,u_I)u_2] = r_2(u_1,\dots,u_I)u_2, \\
&\hspace{2mm}\vdots\\
&\partial_t u_I - \Delta[a_I u_I] = r_I(u_1,\dots,u_I)u_I, 
\end{empheq}
with the initial condition $u_i(0,\cdot) = u_i^{in}$, for all $1\leq i\leq I$. 
\end{thm}
\begin{remark}
Each equation of the above triangular systems is a Kolmogorov equation. The meaning of solution refers to Definition~\ref{def:dist}, which is given in the next section.
\end{remark}
\begin{remark}
When $I=2$, if $u_2^{in}$ is assumed to be bounded, one can get rid of the upper-bound condition on $a_1$ : $u_2$ is bounded by the maximum principle, and the continuity of $a_1$ is sufficient.
\end{remark}
\begin{remark}
  It will be clear from our proof that we could (as in \cite{melfont}) consider different kernels $(\rho_{ij}^n)_{1\leq i,j\leq I}$ for each population and even replace the reaction terms by $r_i(u_1^n\star \theta_{i1}^n,\dots,u_I^n\star \theta_{iI}^n)$ with a different sequence of kernels $(\theta_{ij}^n)$, and obtain the same result when letting all the $(\rho_{ij}^n)_{1\leq i,j\leq I}$ and all $(\theta_{ij}^n)_{1\leq i,j\leq I}$ go to the Dirac mass.
\end{remark}
In \cite{melfont} Fontbona and Méléard proved the above convergence only when the convolution kernels of the reaction terms go to the Dirac mass (the one inside the laplacian are fixed), under stronger regularity assumptions on both the $a_i$ and the $r_i$ functions, but including non-triangular cases and unbounded $a_i$ functions (with however some control on the growth). The domain considered in \cite{melfont} is the whole space space $\R^N$. However, if one has in mind the use of cross-diffusion systems for modeling in population dynamics, the most relevant setting - and the one in which the SKT system was written in the first place (see \cite{Shigesada1979}) - is the initial-boundary value problem (with homogeneous Neumann boundary conditions) on a smooth bounded domain $\Omega\subset\R^N$. Our method of proof for the torus would apply \emph{verbatim} for this more realistic framework, up to minor modifications. The choice of the periodic setting allows to grasp all the essence of the proof while avoiding the tedious task of defining a convolution operator on a domain having a non-empty boundary. Note however that our strategy works identically if one replaces the convolution operator by an elliptic regularization like the one proposed in \cite{LPR} (we chose here the convolution to stay within the framework of \cite{melfont} to which we aim at answering). As a matter of fact, each convolution operator could be replaced by a sequence of operators $A_n:\textnormal{L}^2(Q_T)\rightarrow\textnormal{L}^2(Q_T)$ approaching pointwisely the identity map and preserving strong convergence. Since this extra-generalization does not add any kind of difficulty in our analysis, we prefered to focus on a more digestible statement of our main theorem. 
Another possibility of generalization of our result is of course the case of the whole space $\R^N$ originally considered by Fontbona and Méléard for which one should replace the Poincaré-Wirtinger inequality that we use several times by Sobolev embeddings.

\vspace{2mm}

As detailed earlier, the strength of our result relies strongly on the low regularity of the functions $a_i,r_i$ and the initial data. For instance, the asymptotic study established in \cite{melfont} by Fontbona and Méléard assumes $\mathscr{C}^3$ regularity for the functions $a_i$ and Lipschiz initial data. To ensure that Theorem~\ref{thm:nonloctoclas} is not only an ``if-result'', one needs to prove the existence of the sequence $(u_i^n)_n$ above to make sure that the statement of Theorem~\ref{thm:nonloctoclas} is not empty. This is the purpose of the following result. Up to our knowledge, the current literature on cross-diffusion systems does not furnish a comparable global existence result for such low-regularity coefficients and initial data.
\begin{thm}\label{thm:exi}
Consider two families of continuous real-valued functions $(a_i)_{1\leq i\leq I}$ and $(r_i)_{1\leq i\leq I}$, where for all $i$ the function $a_i$ is defined on $\overline{Q_T}\times \R^{I-i}$ and $r_i$ is defined on $\R^I$. Assume that $a_i$ is bounded from above and below by a positive constant and $r_i$ only from above. Assume furthermore that  $|r_i|$ is sub-affine, that is $|r_i(x_1,\cdots,x_I)|\lesssim 1+|x_1|+\cdots+|x_I|$.  Fix a family of non-negative (and normalized) smooth kernels $(\rho_i)_{1\leq i\leq I}$ and consider a family of $\textnormal{L}^2(\T^N)$ non-negative functions $(u_i^{in})_{1\leq i \leq I}$. Then, there exists a family of $\textnormal{L}^2(Q_T)$ non-negative fuctions $(u_i)_{1\leq i\leq I}$ solution of the following Cauchy problem on $Q_T$
\begin{empheq}[left = \empheqlbrace]{align*}
&\partial_t u_1 - \Delta[a_1(\,\cdot\,,u_2\star\rho_2,u_3\star\rho_3,\dots,u_I\star\rho_I)u_1] = r_1(u_1\star\rho_1,\dots,u_I\star\rho_I)u_1,\\
&\partial_t u_2 - \Delta[a_2(\,\cdot\,,u_3\star\rho_3,\dots,u_I\star\rho_I)u_2] = r_2(u_1\star\rho_1,\dots,u_I\star\rho_I)u_2, \\
&\hspace{2mm}\vdots\\
&\partial_t u_I - \Delta[a_I u_I] = r_I(u_1\star\rho_1,\dots,u_I\star \rho_I)u_I, 
\end{empheq}
with initial condition $u_i(0,\cdot) = u_i^{in}$, for all $1\leq i\leq I$. 
\end{thm} 
   As explained in the introduction, our convergence result (Theorem~\ref{thm:nonloctoclas}) will be obtained step by step, starting with the convergence of $(u_I^n)_n$, we will prove the one of $(u_{I-1}^n)_n$, and then $(u_{I-2}^n)_n$ \emph{etc}. For this reason, almost all the difficulty of the asymptotic is concentrated when focusing on a sequence of scalar diffusion equations, 
\begin{align*}
\partial_t z_n - \Delta (\mu_n z_n) &= R_n z_n,\\
z_n(0,\cdot) &= z^0_n,
\end{align*}
where $z^0_n$ is non-negative, $\mu_n$ and $R_n$ are respectively uniformly positively bounded from below and above, the question being : on which (asymptotic) conditions on $(z^0_n,\mu_n,R_n)_n$ does $(z_n)_n$ enjoys strong compactness properties ? The next result answers (by means of sufficient conditions) to this question. It's a technical tool but its repeated use is at the core of the proof of the two previous theorems so we decided to add it among the main results, since we believe it could be useful in the study of many cross-diffusion systems. It gives sufficient asymptotic condition on $(z_n^0)_n$, $(\mu_n)_n$ and $(R_n)_n$ (namely $(z_n^0)_n\din\textnormal{L}^2(\T^N)$, $(\mu_n)_n\ddin\textnormal{L}^1(Q_T)$ and $(R_n)_n\din \textnormal{L}^2(Q_T)$) to get strong compactness for $(z_n)_n$. 
\begin{lem}\label{lem:comprz}
For all $n\in\N$ consider $z^0_n\in\textnormal{L}^2(\T^N)$ non-negative, two functions $\mu_n,R_n\in \textnormal{L}^\infty(Q_T)$ and $z_n\in\textnormal{L}^2(Q_T)$ a non-negative solution of
\begin{align*}
\partial_t z_n - \Delta(\mu_n z_n) &= R_n z_n,\\
z_n(0,\cdot) &= z^0_n.
\end{align*}
Assume that $\inf_n\inf_{Q_T} \mu_n>0$ and $\sup_n\sup_{Q_T} R_n<\infty$, that  $(z_n^0)_n\din\textnormal{L}^2(\T^N)$, $(R_n)_n\din\textnormal{L}^2(Q_T)$ and that $(\mu_n)_n\ddin\textnormal{L}^1(Q_T)$, with a bounded cluster point. Then one has $(z_n)_n\ddin\textnormal{L}^2(Q_T)$.
\end{lem}
\begin{remark}\label{rem:conv}
The compactness for $(\mu_n)_n$ operates in $\textnormal{L}^1(Q_T)$ but we need the existence of a bounded cluster point in order to use the result of Section \ref{sec:dualsol}. No uniform $\textnormal{L}^\infty(Q_T)$ bounds are required on $(\mu_n)_n$ or $(R_n)_n$, only uniform lower-bound and upper-bound respectively. We give more comments concerning these assumptions in Section~\ref{sec:com}.
\end{remark}
When $\mu_n$ equals a positive constant and $R_n=0$, the previous lemma is nothing more than the standard compactness of the heat operator, a fact that could be seen directly by gradient estimate or using the heat kernel for instance. However, it is important to note that, still with $R_n=0$, even the case $\mu_n=\mu\in\textnormal{L}^\infty(Q_T)$ (which follows directly from our result) does not seem obvious to be established directly by energy or kernel methods.

\medskip 

If $R_n z_n$ is replaced by a sequence $G_n$, no assumption of the sign of $z_n$ is required and compactness can be replaced by stability (see Proposition~\ref{propo:comp}), because the Kolmogorov equation is then well-posed. A similar result is possible if $R_n$ depends solely on the time variable (by a suitable change of unknown), but we do not detail this type of generalization herein.
\section{The Kolmogorov equation}\label{sec:dualsol}
In this section we focus on the problem \eqref{intro:eq:kol1} -- \eqref{intro:eq:kol2}, for which we start by giving a precise of solution on the periodic parabolic cylinder $Q_T$.
\begin{definition}\label{def:dist}
For any $G\in\textnormal{L}^1(Q_T)$,  $z_0\in\textnormal{L}^1(\T^N)$ and any measurable function $\mu$ on $Q_T$, a function $z\in\textnormal{L}^1(Q_T)$ is a distributional solution of the Cauchy problem (or simply: a solution in $\mathscr{D}'(Q_T)$) \eqref{intro:eq:kol1} -- \eqref{intro:eq:kol2} if $z\mu\in\textnormal{L}^1(Q_T)$ and furthermore 
\begin{align}
\label{eq:def:dist}\forall\ffi\in\mathscr{D}(Q_T),\qquad -\int_{Q_T} z(\partial_t \ffi +\mu\Delta\ffi) = \int_{\T^N} z^0 \ffi(0,\cdot) + \int_{Q_T} G\ffi.
\end{align}
\end{definition}
Our analysis of the Kolmogorov equation relies strongly on the study of the dual problem which is introduced in the next paragraph.
\subsection{Generic results on the dual problem} 
We focus in this paragraph on the so-called \emph{dual problem}
\begin{align}
\label{eq:dua1}\partial_t \Phi + \mu\Delta \Phi &= S, \\
\label{eq:dua2}\Phi(T,\cdot) &= 0.
\end{align}
Pierre, Martin and Schmitt were the first to remark \cite{PiSc,MaPi} that one can recover valuable information on the Kolmogorov equation from the dual one. In these articles the diffusion coefficient $\mu$ was considered constant and the main information recovered on the solution $z$ of the Kolmogorov equation was the so-called duality lemma that we will encounter in the next paragraph, as a by-product of our analysis. We propose here to go further into the exploration of the dual problem, in order to transfer as much possible information to the original Kolmogorov equation.

\medskip 

\noindent We first have the following well-posedness result.
\begin{Propo}\label{propo:duaphi}
Consider $\mu\in\textnormal{L}^\infty(Q_T)$ positively lower-bounded. For all $S\in\textnormal{L}^2(Q_T)$ there exists a unique fonction 
\begin{align*}
\Phi\in \textnormal{L}^\infty(0,T;\H^1(\T^N))\cap\textnormal{L}^2(0,T;\H^2(\T^N))\cap\mathscr{C}^0([0,T];\textnormal{L}^2(\T^N)),
\end{align*}
solution of \eqref{eq:dua1} -- \eqref{eq:dua2}. This solution satisfies the following \emph{a priori} estimates
\begin{align}
\label{ineq:apriori1}\|\nabla \Phi \|_{\textnormal{L}^\infty(0,T;\textnormal{L}^2(\T^N))}^2 + \|\mu^{1/2}\Delta \Phi\|_{\textnormal{L}^2(Q_T)}^2 &\leq \|\mu^{-1/2}S\|_{\textnormal{L}^2(Q_T)}^2,\\
\label{ineq:apriori2}\| \Phi \|^2_{\textnormal{L}^\infty(0,T;\textnormal{L}^2(\T^N))} &\leq C_N \left(\|\mu\|_{\textnormal{L}^1(Q_T)}+1 \right) \|\mu^{-1/2}S\|_{\textnormal{L}^2(Q_T)}^2.
\end{align}
The mapping $S\mapsto \Phi$  is linear continuous from $\textnormal{L}^2(Q_T)$ to $\textnormal{L}^2(0,T;\H^2(\T^N))\cap \textnormal{L}^\infty(0,T;\H^1(\T^N))$, and compact from $\textnormal{L}^2(Q_T)$ to $\mathscr{C}^0([0,T];\textnormal{L}^2(\T^N))$. Furthermore, if $S\geq 0$ we have $\Phi\leq 0$. 
\end{Propo}
\begin{remark}
Equations \eqref{eq:dua1} and  \eqref{eq:dua2} hold in $\textnormal{L}^2(\T^N)$.
\end{remark}
\begin{remark}\label{rem:obs}
The main obstruction to the generalization of Theorem~\ref{thm:nonloctoclas} without the assumption of boundedness of the diffusion coefficients is the establishment of a similar well-posedness result assuming for instance that $\mu$ is (only) integrable, see Section \ref{sec:com} for more details.
\end{remark}
\begin{proof}
  The core of the proof lies in the \emph{a priori} estimates \eqref{ineq:apriori1} -- \eqref{ineq:apriori2}  which are automatically satisfied as soon as $\Phi$ satisfies \eqref{eq:dua1} -- \eqref{eq:dua2} and has the announced regularity. Indeed, in that case since $S\in\textnormal{L}^2(Q_T)$ and $\Delta\Phi\in\textnormal{L}^2(Q_T)$, the equation \eqref{eq:dua1} ensures that $\partial_t \Phi\in\textnormal{L}^2(Q_T)$ because $\mu$ is bounded (if it was not the case we would only have $\mu^{-1/2}\partial_t \Phi\in\textnormal{L}^2(Q_T)$). The equation being satisfied pointwisely we can multiply it by $\Delta \Phi$ and integrate on $[t,T]\times\T^N$ to get by Cauchy-Scwharz and Young's inequalities
\begin{align*}
\int_t^T\int_{\T^N} \Delta\Phi\,\partial_t \Phi  + \frac{1}{2}\|\mu^{1/2}\Delta\Phi\|_{\textnormal{L}^2([t,T]\times\T^N)}^2 \leq \frac{1}{2}\|\mu^{-1/2}S\|_{\textnormal{L}^2([t,T]\times\T^N])}^2.
\end{align*}
If $\Phi\in\mathscr{C}^0([0,T];\H^1(\T^N))$, a simple integration by parts would replace the bracket in the left hand side by $\|\nabla \Phi(t)\|_{\textnormal{L}^2(\T^N)}^2/2$. To overcome this difficulty we can obtain estimate \eqref{ineq:apriori1} by an adequate approximation argument (convolution) considering a sequence of regular functions $\Phi_k$ such that 
\begin{itemize}
\item[$\bullet$] $(\partial_t \Phi_k,\Delta \Phi_k)_k$ approaches $(\partial_t \Phi,\Delta \Phi)$ in $\textnormal{L}^2(Q_T)\times\textnormal{L}^2(Q_T)$\,;
\item[$\bullet$] $\Phi_k(T,\cdot)=0$ in $\textnormal{L}^2(\T^N)$\,;
\item[$\bullet$] $\|\nabla \Phi\|_{\textnormal{L}^\infty(0,T;\textnormal{L}^2(\T^N))} \leq \underline{\lim}_k \|\nabla \Phi_k\|_{\textnormal{L}^\infty(0,T;\textnormal{L}^2(\T^N))}$.
\end{itemize}
Introducing $S_k=\partial_t \Phi_k + \mu \Delta \Phi_k$, we have that $(S_k)_k$ converges in $\textnormal{L}^2(Q_T)$ toward $S$ and when $k$ is fixed, the previous integration by parts is legitimate so we may write for all $t\in[0,T]$
\begin{align*}
\frac{1}{2}\|\nabla \Phi_k(t)\|_{\textnormal{L}^2(\T^N)}^2 + \frac{1}{2}\|\mu^{1/2}\Delta\Phi_k\|_{\textnormal{L}^2([t,T]\times\T^N)}^2 = \frac{1}{2}\|\mu^{-1/2}S_k\|_{\textnormal{L}^2([t,T]\times\T^N)}^2.
\end{align*}
In particular we have the following estimate 
\begin{align*}
\frac{1}{2}\|\nabla \Phi_k\|_{\textnormal{L}^\infty(0,T;\textnormal{L}^2(\T^N)}^2 +\frac{1}{2}\|\mu^{1/2}\Delta \Phi_k\|_{\textnormal{L}^2(Q_T)}^2 \leq \frac{1}{2}\|\nabla \Phi_k(T)\|_{\textnormal{L}^2(\T^N)}^2 + \frac{1}{2}\|\mu^{-1/2}S_k\|_{\textnormal{L}^2(Q_T)},
\end{align*}
from which we infer \eqref{ineq:apriori1}. Integrating \eqref{eq:dua1} on $[t,T]\times\T^N$ and using $\Phi(T)=0$ we infer
\begin{align*}
\int_{\T^N} \Phi(t) = \int_t^T \int_{\T^N}\Big(\mu\Delta \Phi- S\Big),
\end{align*}
from which we get by Cauchy-Schwarz inequality: 
\begin{align*}
\left|\int_{\T^N} \Phi(t)\right| \leq \|\mu\|_{\textnormal{L}^1(Q_T)}^{1/2} \Big(\|\mu^{1/2}\Delta \Phi\|_{\textnormal{L}^2(Q_T)} + \|\mu^{-1/2}S\|_{\textnormal{L}^2(Q_T)}\Big),
\end{align*}
and we recover estimate \eqref{ineq:apriori2} using estimate \eqref{ineq:apriori1} and the Poincaré-Wirtinger inequality.

\vspace{2mm} 

Note that the uniqueness of $\Phi$ (in the given functionnal spaces) is a consequence of estimates \eqref{ineq:apriori1} -- \eqref{ineq:apriori2} that we just proved.

\vspace{2mm}

It remains to justify the existence of $\Phi$ and the properties of the mapping $S\mapsto \Phi$. Pick a regularization $(\mu_k)_k\in\mathscr{C}^\infty(Q_T)$ of $\mu$, approaching it in $\textnormal{L}^1(Q_T)$ and (uniformly) positively lower-bounded by a  $\alpha>0$. Consider also $(S_k)_k$ a sequence of smooth functions approaching $S$ (in $\textnormal{L}^2(Q_T)$). The operator
\begin{align*}
L_k:=\partial_t + \mu_k \Delta   = \partial_t + \textnormal{div}(\mu_k \nabla) - \nabla \mu_k \cdot \nabla  
\end{align*}
is then uniformly (backward) parabolic : there exists a unique (smooth) $\Phi_k$ satisfying 
\begin{align*}
L_k \Phi_k = S_k,
\end{align*}
together with the terminal condition $\Phi_k(T,\cdot)= 0$. This solution is non-positive as soon as $S_k\geq 0$, by the maximum principle. We now use the \emph{a priori} estimates that we proved before: since $(\mu_k)_k \din\textnormal{L}^1(Q_T)$ we have $(\Phi_k)_k \din \textnormal{L}^\infty(0,T;\H^1(\T^N))$ and $(\mu_k^{1/2}\Delta \Phi_k)_k\din\textnormal{L}^2(Q_T)$. Since $\mu_k \geq \alpha$, the latter estimate leads to $(\Delta \Phi_k)_k\din\textnormal{L}^2(Q_T)$ so that we have eventually $(\Phi_k)_k\din\textnormal{L}^2(0,T;\H^2(\T^N))$. In particular, using the equality $\partial_t \Phi_k = -\mu_k \Delta \Phi_k + S_k$, we get $(\partial_t \Phi_k)_k\din\textnormal{L}^2(Q_T)$, whence uniform $\textnormal{L}^2$-equicontinuity (in time) for $(\Phi_k)_k$. Since this sequence is also bounded in $\textnormal{L}^\infty(0,T;\H^1(\T^N))$, we infer by Rellich's and Ascoli's theorems that $(\Phi_k)_k\ddin \mathscr{C}^0([0,T];\textnormal{L}^2(\T^N))$. We consider then a subsequence converging (whether it is strongly, weakly or weakly$-\star$) in all the previous spaces. The corresponding cluster point is the solution $\Phi$ and is indeed non-positive if the $S_k$ are chosen like this, which can indeed be done if $S\geq 0$. Lastly, the continuity of the mapping $S\mapsto \Phi$ is clear : all the previous estimates can be written in the form $\lesssim \|S\|_{\textnormal{L}^2(Q_T)}$ (because $\mu$ is lower-bounded). As for the compactness of the mapping with value in $\mathscr{C}^0([0,T];\textnormal{L}^2(\T^N))$, it is obtained as before using Rellich's and Ascoli's theorems. $\qedhere$
\end{proof}
Contrary to the Kolmogorov equation for which the stability of the solution with respect to $\mu$ will demand more effort, a similar statement (in fact, a stronger one) for the dual problem is rather easy to prove.
\begin{lem}\label{lem:comphi}
For all $n\in\N$ consider $\mu_n\in\textnormal{L}^\infty(Q_T)$ uniformly (in $n,t,x$) lower-bounded by a positive constant, $S_n\in\textnormal{L}^2(Q_T)$ and $\Phi_n$ the solution of the dual problem (given by Proposition \ref{propo:duaphi})
\begin{align*}
  \partial_t \Phi_n + \mu_n\Delta \Phi_n &= S_n, \\
\Phi_n(T,\cdot) &= 0.
\end{align*}
If $(\mu_n)_n\ddin\textnormal{L}^1(Q_T)$ (for the weak topology) and $(S_n)_n\din\textnormal{L}^2(Q_T)$, then $(\Phi_n)_n\ddin\mathscr{C}^0([0,T];\textnormal{L}^2(\T^N))$.
\end{lem}
\begin{proof}
Thanks to Proposition \ref{propo:duaphi}, since $(\mu_n)_n\din\textnormal{L}^1(Q_T)$ and $(\mu_n)_n$ is uniformly lower-bounded, we have $(\Phi_n)_{n}\din\textnormal{L}^\infty(0,T;\H^1(\T^N))$ and similarly $(\mu_n^{1/2}\Delta\Phi_n)_{n}\din\textnormal{L}^2(Q_T)$. We can now reproduce the argument used in the proof of Proposition \ref{propo:duaphi}. The crucial point is to notice that since $(\mu_n)_n\ddin\textnormal{L}^1(Q_T)$ for the weak topology, it is uniformly integrable. Therefore, using the Cauchy-Schwarz inequality we get that both $(\mu_n \Delta \Phi_n)_n$ and $(\mu_n^{1/2} S_n)_n$ are also uniformly integrable. Eventually, we deduce the equi-continuity of  $(\Phi_n)_n$ in $\mathscr{C}^0([0,T];\textnormal{L}^1(\T^N))$ and then its compactness in $\mathscr{C}^0([0,T];\textnormal{L}^2(\T^N))$ as we did in the proof of Proposition \ref{propo:duaphi}. $\qedhere$
\end{proof}
\subsection{Dual solutions}
The assumption of boundedness on $\mu$ gives a framework in which the problem \eqref{intro:eq:kol1} -- \eqref{intro:eq:kol2} is well-posed. 
\begin{thm}\label{thm:duaz}
Fix $\mu\in\textnormal{L}^\infty(Q_T)$ positively lower-bounded, $G\in \textnormal{L}^2(Q_T)$ and $z^0\in\textnormal{L}^2(\T^N)$. There exists a unique solution $z$ to equations \eqref{intro:eq:kol1} -- \eqref{intro:eq:kol2} in the sense of Definition \ref{def:dist}. It satisfies the (stronger) formulation
\begin{align}\label{eq:form:strong}
\forall S\in\textnormal{L}^2(Q_T),\qquad -\int_{Q_T} z S = \int_{\T^N} z^0 \Phi_\mu^S(0) + \int_{Q_T} G \Phi_\mu^S,
\end{align}
 and the following estimate 
\begin{align}
\label{ineq:z}\|\mu^{1/2}z\|_{\textnormal{L}^2(Q_T)} \leq C_N \Big(\|\mu\|_{\textnormal{L}^1(Q_T)}^{1/2}+1\Big)\Big(\|z^0\|_{\textnormal{L}^2(\T^N)}+\|G\|_{\textnormal{L}^1(0,T;\textnormal{H}^{-1}(\T^N)))}\Big).
\end{align}
\end{thm}
\begin{remark}
As suggested by estimate \eqref{ineq:z}, it is possible to extend considerably the admissible function space for the source term $G$ (and in fact, also the initial data $z^0$), but we avoid here this generalization for the sake of simplicity.
\end{remark}
\begin{proof}
Let us first prove the \emph{a priori} estimate \eqref{ineq:z}. If $z$ solves \eqref{intro:eq:kol1} -- \eqref{intro:eq:kol2}, we have by Definition~\ref{def:dist}
\begin{align}\label{eq:proofthz}
-\int_{Q_T} z(\partial_t \ffi + \mu \Delta\ffi) = \int_{\T^N} z^0 \ffi(0,\cdot) + \int_{Q_T} G \ffi,
\end{align}
for all test functions $\ffi$ in $\mathscr{D}(Q_T)$. Introducing $S:=\partial_t \ffi + \mu\Delta \ffi$ we have obviously $\Phi_\mu^{S} = \ffi$ and the estimate of Proposition~\ref{propo:duaphi} thus leads to
\begin{align*}
\left| \int_{Q_T} zS\right| &\leq \|z^0\|_{\textnormal{L}^2(\T^N)} \|\Phi^S_\mu(0)\|_{\textnormal{L}^2(\T^N)} + \|\Phi^S_\mu\|_{\textnormal{L}^\infty(0,T;\H^1(\T^N))}\|G\|_{\textnormal{L}^1(0,T;\H^{-1}(\T^N))} \\
                            &\leq C_N \Big(\|\mu\|_{\textnormal{L}^1(Q_T)}^{1/2}+1\Big)\Big(\|z^0\|_{\textnormal{L}^2(\T^N)}+\|G\|_{\textnormal{L}^1(0,T;\textnormal{H}^{-1}(\T^N)))}\Big)\|\mu^{-1/2}S\|_{\textnormal{L}^2(Q_T)}.
\end{align*}
Therefore, if $\big\{ \mu^{-1/2}\partial_t \ffi +\mu^{1/2}\Delta \ffi\,:\,\ffi\in\mathscr{D}(Q_T)\big\}$ is dense in $\textnormal{L}^2(Q_T)$, estimate \eqref{ineq:z} will follow directly by a duality argument. Herein the possible complications (see Section~\ref{sec:com}) are annihilated by the boundedness assumption on $\mu$ : the previous density is a simple consequence of the one of $\big\{(\partial_t \ffi,\Delta \ffi)\,:\,\ffi\in\textnormal{L}^2(Q_T)\big\}$ in $\textnormal{L}^2(Q_T)\times\textnormal{L}^2(Q_T)$ which is immediately obtained by standard approximation. Estimate \eqref{ineq:z} is thus proved. 

\medskip 

Starting over from \eqref{eq:proofthz}, the strong formulation \eqref{eq:form:strong} follows by an approximation similar to the one used in the proof of Proposition~\ref{propo:duaphi} : for any $S\in\textnormal{L}^2(Q_T)$, we can (by convolution) pick a sequence $(\Phi_k)_k$ in $\mathscr{D}(Q_T)$ approaching $\Phi_\mu^S$ in $\mathscr{C}^0([0,T];\textnormal{L}^2(\T^N))$ and such that $(\partial_t \Phi_k,\Delta\Phi_k)_k$ approaches $(\partial_t \Phi_\mu^S,\Delta\Phi_\mu^S)$ in $\textnormal{L}^2(Q_T)\times\textnormal{L}^2(Q_T)$. This allows to establish \eqref{eq:form:strong}.

\medskip

It remains to prove that actually such a solution $z$ exists and that it's unique. Because of the previous \emph{a priori} analysis, it is sufficient to establish this well-posedness in $\textnormal{L}^2(Q_T)$. Thanks to the estimate of Proposition~\ref{propo:duaphi}, the linear map
\begin{align*}
\textnormal{L}^2(Q_T) &\longrightarrow \R\\
S&\longmapsto -\int_{\T^N} z^0\Phi^S_\mu(0) - \int_{Q_T} \Phi^S_\mu  G,
\end{align*}
 is continuous. Owing to Riesz's representation theorem we infer the existence of unique element $z\in\textnormal{L}^2(Q_T)$ solving \eqref{eq:form:strong} and the proof is over. $\qedhere$\end{proof}
The inequality \eqref{ineq:z} is known in the literature as the ``duality estimate''. We are stating it in a particular setting in which the diffusion is bounded which allows to solve both the Kolmogorov and the dual equation uniquely. But the strength of the duality estimate is that it can be written as soon as the dual equation has \emph{one} solution and it can hence be used as an \emph{a priori} estimate for any solution of the Kolmogorov equation when the diffusion $\mu$ is only supposed integrable (and positively lower-bounded). This for instance was used in the context of cross-diffusion systems in the articles \cite{DesLepMou,dlmt,lepmou}. However, it is important to note that in all these references the duality lemma was either stated in an idealized setting (assuming $z$ to be regular enough to justify all the manipulations) or a discretized one. In our situation we use rigorously the duality lemma for rather weak solutions, the counterpart being this boundedness assumption on the diffusion coefficient. Note that in this framework a stronger form of the duality estimate is available, allowing for $\textnormal{L}^p(Q_T)$ (with $p>2$) control on the solution, see \cite{dualimpro}. Here, our purpose is to exhibit stability and/or compactness of sequences of solutions to the Kolmogorov equation under this particular setting.
\begin{definition}\label{defi:dualsol}
We say that $z$ is the \textsf{dual solution} of problem \eqref{intro:eq:kol1} -- \eqref{intro:eq:kol2} when we are under the assumptions of Theorem~\ref{thm:duaz}, $z$ being the solution given by this theorem.
\end{definition}
Not only the dual solutions have a stronger variational formulation, but they also enjoy a kind of weak maximum principle. When $\mu$ is regular (say, up to the two first derivatives in $x$), this is not surprising because  $\Delta(\mu z)$ can be written $\nabla\cdot (\mu \nabla z) + \nabla z \cdot \nabla \mu + z \Delta \mu$, so that (weak and strong) maximum principles are simply consequence of the underlying parabolic structure. The situation we consider here is a lot worse, since $\mu$ is only assumed bounded. The following proposition will be useful for the reaction terms that we consider in our cross-diffusion systems.
\begin{Propo}\label{propo:sublin}
Consider $z$ the dual solution of the Cauchy problem \eqref{intro:eq:kol1} -- \eqref{intro:eq:kol2}. We have (disjointly)
\begin{itemize}
\item[$(i)$] If $z^0$ and $G$ are non-negative, then so is $z$ ;
\item[$(ii)$] If $G\leq rz $ for some $r\in\R$, then $z(t,x)\leq \widetilde{z}(t,x)e^{rt}$, where $\widetilde{z}$ is the dual solution of
\begin{align*}
\partial_t \widetilde{z} - \Delta (\mu \widetilde{z}) &= 0,\\
\widetilde{z}(0,\cdot) = z^0.
\end{align*}
\end{itemize}
\end{Propo} 
\begin{proof}
We start by $(i)$. We could argue by approximation (since we have uniqueness), but we use instead the classical argument of Stampacchia, rephrasing it in our setting. If $z^-$ is the negative part of $z$, we have $z^-\in\textnormal{L}^2(Q_T)$ so it is admissible in the formulation \eqref{eq:form:strong} and hence, denoting by $\langle \cdot,\cdot\rangle$ the $\textnormal{L}^2(Q_T)$ or $\textnormal{L}^2(\T^N)$ inner-product, we have $-\langle z,z^-\rangle = \langle z^0,\Phi^{z^-}_\mu(0)\rangle + \langle \Phi_\mu^{z^-},G\rangle$. Since $z^-\geq 0$, using Proposition \ref{propo:duaphi} we see that $\Phi_\mu^{z-} \leq 0$, since $z^0 \geq 0$ and $G \geq 0$, we eventually obtained $-\langle z,z^-\rangle \leq 0$ which is nothing less than $\|z^-\|_{\textnormal{L}^2(Q_T)}^2 \leq 0$, \emph{i.e.} $z^-=0$ and $z\geq 0$. 

\medskip

\noindent For $(ii)$, an easy computation shows that $w(t,x) := e^{-rt} z(t,x)$ is the dual solution of 
\begin{align*}
\partial_t w -\Delta(\mu w) &= (G - r z) e^{-rt},\\
w(0) &= z^0.
\end{align*} 
If $\widetilde{z}$ is the dual solution of 
\begin{align*}
\partial_t \widetilde{z} - \Delta(\mu \widetilde{z} ) &= 0,\\
\widetilde{z}(0)&=z^0,
\end{align*}
by linearity one checks that $\xi:=\widetilde{z}-w$ is the dual solution of 
\begin{align*}
\partial_t \xi - \Delta(\mu \xi) = (rz-G) e^{-rt}.
\end{align*}
Since the r.h.s. is non-negative (and the initial data is $0$), we infer from the point $(i)$ that we just prove that $\xi \geq 0$, which is exactly what we wanted to prove. $\qedhere$
\end{proof}

\subsection{A stability result}
The following proposition asserts that, for dual solutions, strong convergence of $\mu_n$ (without any control of its gradient) can be transfered to the corresponding sequence $z_n$. It is the key-tool that we use to prove Lemma~\ref{lem:comprz}. The well-posedness of both the dual and the Kolmogorov equation plays a crucial role in this stability result.
\begin{Propo}\label{propo:comp}
For all $n\in\N$ consider $z_n\in\textnormal{L}^2(Q_T)$, the dual solution (in the sense of Definition~\ref{defi:dualsol}) of 
\begin{align}
\label{eq:zn1}\partial_t z_n -\Delta (\mu_n z_n) &= G_n, \\
\label{eq:zn2}z_n(0,\cdot) &= z^0_n.
\end{align}
Assume the existence of $z^0\in\textnormal{L}^2(\T^N)$, $G\in\textnormal{L}^2(Q_T)$ and $\mu\in\textnormal{L}^\infty(Q_T)$ such that $(z^0_n)_n\rightharpoonup z^0$ in $\textnormal{L}^2(\T^N)$, $(G_n)_n\rightharpoonup G$ in $\textnormal{L}^2(Q_T)$ and $(\mu_n)_n\rightarrow \mu$ in $\textnormal{L}^1(Q_T)$. Then $(\mu_n^{1/2} z_n)_n$ and $(z_n)_n$ converge in $\textnormal{L}^2(Q_T)$ to respectively $\mu^{1/2}z$ and $z$, where $z$ is the dual solution of the problem
\begin{align*}
\partial_t z - \Delta(\mu z) &= G,\\
z(0,\cdot) &= z^0.
\end{align*}
\end{Propo}
\begin{proof}
Using estimate \eqref{ineq:z} of Theorem~\ref{thm:duaz}, we get first $(\mu_n^{1/2}z_n)_n \din\textnormal{L}^2(Q_T)$ and because $(\mu_n)_n$ is uniformly (in $n$) positively lower-bounded, we have in turn $(z_n)_n \din\textnormal{L}^2(Q_T)$. On the other hand, since $(\mu_n)_n\rightarrow \mu$ in $\textnormal{L}^1(Q_T)$ we have $(\mu_n^{1/2})_n \rightarrow \mu^{1/2}$ in $\textnormal{L}^2(Q_T)$ (use for instance that $x\mapsto\sqrt{x}$ is Lipschitz away from $0$). Omitting the extractions, we can thus write (using weak/strong convergence product for the two last lines)
\begin{align*}
(\mu_n^{1/2})_n &\rightarrow \mu^{1/2}, \text{ in }\textnormal{L}^2(Q_T),\\
(z_n)_n&\rightharpoonup z, \text{ in }\textnormal{L}^2(Q_T),\\
(\mu_n^{1/2}z_n)_n&\rightharpoonup \mu^{1/2}z, \text{ in }\textnormal{L}^2(Q_T),\\
(\mu_n z_n)_n&\rightharpoonup \mu z, \text{ in }\textnormal{L}^1(Q_T).
\end{align*}
Let us now notice that $(\mu_n^{1/2} z_n)_n\ddin \textnormal{L}^2(Q_T) \Rightarrow (z_n)_n\ddin\textnormal{L}^2(Q_T)$. Indeed, if (a subsequence of) $(\mu_n^{1/2} z_n)_n$ converges in $\textnormal{L}^2(Q_T)$, it is uniformly integrable and this property is transfered to (the corresponding subsequence of) $(z_n)_n$, because $\mu_n$ is uniformly positively lower-bounded. Up to an extraction we can assume that $(\mu_n)_n$ and $(\mu_n^{1/2}z_n)_n$ converge a.e. and since $(\mu_n)_n$ is uniformly positively lower-bounded, we get that $(z_n)_n$ converges a.e. : we recover convergence in $\textnormal{L}^2(Q_T)$ for $(z_n)_n$ thanks to Vitali's convergence Theorem. 

\vspace{2mm}

It is hence sufficient to prove that $(\mu_n^{1/2} z_n)_n \ddin\textnormal{L}^2(Q_T)$ and that $(z_n)_n$ have only one possible (strong) limit point in $\textnormal{L}^2(Q_T)$. For the latter fact, if $z$ is as above a weak limit point of $(z_n)_n$, then it is actually the dual solution of the following problem
\begin{align}
\label{eq:z1comp}\partial_t z - \Delta(\mu z ) &= G, \\
\label{eq:z2comp}z(0,\cdot) &= z^0.
\end{align}
Indeed, the above list of convergences (and more crucially the last one) allow to pass to the limit \eqref{eq:zn1} -- \eqref{eq:zn2} and prove that $z$ solves \eqref{eq:z1comp} -- \eqref{eq:z2comp} in the distributional sense, which, thanks to Theorem~\ref{thm:duaz}, is sufficient to identify completely $z\in\textnormal{L}^2(Q_T)$.

\medskip

So let us prove that $(\mu_n^{1/2} z_n)_n\ddin\textnormal{L}^2(Q_T)$. For all $n\in\N$, by definition of dual solutions, we have $z_n\in\textnormal{L}^2(Q_T)$, so that  $\mu_n z_n \in\textnormal{L}^2(Q_T)$ : we can thus consider $\Phi_n$ the unique solution given by Proposition \ref{propo:duaphi} of the following problem
\begin{align*}
\partial_t \Phi_n + \mu_n \Delta\Phi_n &= \mu_n z_n,\\
\Phi_n(T,\cdot) &= 0.
\end{align*}
Using first $(\mu_n)_n\din\textnormal{L}^1(Q_T)$ and $(\mu_n^{1/2} z_n)_n\din\textnormal{L}^2(Q_T)$, we  get by estimates \eqref{ineq:apriori1} -- \eqref{ineq:apriori2} of Proposition \ref{propo:duaphi} that  $(\Phi_n)_n\din\textnormal{L}^\infty(0,T;\H^1(\T^N))$ and $(\mu_n^{1/2} \Delta \Phi_n)_n \din\textnormal{L}^2(Q_T)$ so that we have in particular $(\Phi_n)_n\din\textnormal{L}^2(0,T;\H^2(\T^N))$. By assumption we have $(\mu_n)_n\ddin\textnormal{L}^1(Q_T)$, so that we infer from Lemma \ref{lem:comphi} that ${(\Phi_n)_n\ddin\mathscr{C}^0([0,T];\textnormal{L}^2(\T^N))}$, and also that $(\mu_n^{1/2}\Delta\Phi_n)_n\din\textnormal{L}^2(Q_T)$. The strong convergence of $(\mu_n)_n$ in $\textnormal{L}^1(Q_T)$ and the weak one of $(\mu_n z_n)_n$ towards $\mu z$ (see the list of convergences established above)  allow to identify a.e. the limit of each terms of the equation defining $\Phi_n$. In the same way, any weak limit point $\Phi$ of $(\Phi_n)_n$ in $\textnormal{L}^2(0,T;\H^2(\T^N))$ satisfies $\mu^{1/2}\Delta \Phi \in\textnormal{L}^2(Q_T)$. At the end of the day, there is only one possible cluster point for $(\Phi_n)_n$ in $\mathscr{C}^0([0,T]:\textnormal{L}^2(\T^N))$: it's the unique solution $\Phi$ given by Proposition \ref{propo:duaphi} of the problem (we use here that $\mu$ is assumed to be bounded)
\begin{align*}
\partial_t \Phi + \mu\Delta \Phi &= \mu z, \\
\Phi(T,\cdot)&=0.
\end{align*}
So in fact, we have that the whole sequence $(\Phi_n)_n$ converges to $\Phi$ in $\mathscr{C}^0([0,T];\textnormal{L}^2(\T^N))$. Now we use $\mu_n z_n$ as test function in the definition of $z_n$ as dual solution to \eqref{eq:zn1} -- \eqref{eq:zn2}. We thus have
\begin{align}
\label{eq:endcompduan}- \int_{Q_T} z_n^2\mu_n  = \int_{\T^N} z_n^0 \Phi_n(0)  + \int_{Q_T} \Phi_nG_n.
\end{align}
Similarly, the equality defining $z$ as dual solution of the limit equation (with $\mu^{1/2}z$ as test function) writes
\begin{align}
\label{eq:endcompduaz}- \int_{Q_T} z^2 \mu  = \int_{\T^N} z^0 \Phi(0)  + \int_{Q_T} \Phi\, G.
\end{align}
The r.h.s. of \eqref{eq:endcompduan} converges to the r.h.s. of \eqref{eq:endcompduaz} because of the convergences that we proved for $(\Phi_n)_n$. We thus obtain norm convergence of $(\mu_n^{1/2} z_n)_n$ towards $\mu^{1/2}z$ in $\textnormal{L}^2(Q_T)$ and the proof is over. $\qedhere$\end{proof}

\section{Proof of the main results}\label{sec:proofth}
\subsection{A compactness result for Kolmogorov equation}\label{subsec:comp}
\begin{proof}[Proof of Lemma~\ref{lem:comprz}]
Let us first notice that due to the assumptions, $z_n$ is in fact the dual solution (recall Definition~\ref{defi:dualsol}) of
\begin{align*}
\partial_t z_n - \Delta(\mu_n z_n) &= R_n z_n,\\
 z_n(0,\cdot) &=z^0_n.
\end{align*}
We cannot however directly use the stability result given by Proposition \ref{propo:comp} because nothing ensures that $(R_n z_n)_n\din\textnormal{L}^2(Q_T)$. In fact, at this stage we don't even have any kind of uniform bound on the sequence $(z_n)_n$, but this is solved easily with the help of Proposition~\ref{propo:sublin}. Indeed, the sequence $(R_n)_n$ is assumed to be uniformly upper-bounded say by a positive constant $\rho$. We thus infer from point $(ii)$ of Proposition \ref{propo:sublin} that $z_n \leq \widetilde{z}_n e^{\rho T}$ where $\widetilde{z}_n$ is the dual solution of 
\begin{align*}
\partial_t \widetilde{z}_n - \Delta(\mu_n \widetilde{z}_n) &= 0,\\
 z_n(0,\cdot) &=z^0_n.
\end{align*}
Now we can invoke Proposition \ref{propo:comp} for this last equation (because  $(z_n^0)_n\din\textnormal{L}^2(\T^N)$ and $(\mu_n)_n\ddin\textnormal{L}^1(Q_T)$ with a bounded cluster point) and get first that $(\mu_n\widetilde{z}_n)_n \ddin\textnormal{L}^2(Q_T)$. Since $z_n$ is assumed non-negative, estimate $z_n\leq \widetilde{z}_n e^{\rho T}$ ensures that $(z_n)_n\din\textnormal{L}^2(Q_T)$ and even more: $(z_n)_n$ is uniformly $\textnormal{L}^2(Q_T)$ integrable. Now, let us decompose (by uniqueness of dual solutions) $z_n = z_n^++z_n^-$, where $z_n^+$ and $z_n^-$ are respectively the dual solutions of the two following problems 
\begin{align*}
\partial_t z_n^+ - \Delta(\mu_n z_n^+) &= \rho z_n,\\
z_n^+(0,\cdot) &= z^0_n,
\end{align*}
and 
\begin{align*}
\partial_t z_n^- - \Delta(\mu_n z_n^-) &= (R_n-\rho) z_n,\\
z_n^-(0,\cdot) &= 0.
\end{align*}
Since $z_n$ and $z^0_n$ are non-negative and $R_n\leq \rho$ by assumption, we have $z_n^- \leq 0 \leq z_n^+$. We already noticed that $(z_n)_n\din\textnormal{L}^2(Q_T)$, so that Proposition \ref{propo:comp} ensures $(\mu_n z_n^+)_n\ddin\textnormal{L}^2(Q_T)$ which implies in particular (since $\mu_n$ is uniformly lower-bounded) $(z_n^+)_n \ddin\textnormal{L}^2(Q_T)$. To conclude the  proof it remains thus to prove $(z_n^-)_n \ddin\textnormal{L}^1(Q_T)$, because we already have uniform integrability in $\textnormal{L}^2(Q_T)$ of $(z_n)_n$: the strong compactness in $\textnormal{L}^2(Q_T)$ will then follow from Vitali's convergence Theorem.

\vspace{2mm}

Since $(R_n-\rho)_n \din\textnormal{L}^2(Q_T)$ and $(z_n)_n$ is uniformly $\textnormal{L}^2(Q_T)$ integrable, setting $P_n := (R_n-\rho)z_n$, the sequence $(P_n)_n$ is uniformly integrable, so that 
\begin{align}
\label{unifint}\sup_{n\in\N} \|\mathds{1}_{|P_n| > p} P_n\|_{\textnormal{L}^1(Q_T)} \operatorname*{\longrightarrow}_{p\rightarrow +\infty} 0.
\end{align}
Now, for any fixed $p\in\N$ we can decompose (again by uniqueness of dual solutions) $z_n^- = u_n^p + v_n^p$, where $u_n^p$ and $v_n^p$ are respectively the dual solutions of the two following problems 
\begin{align*}
\partial_t u_n^p - \Delta(\mu_n u_n^p) &= \mathds{1}_{|P_n|\leq p} P_n,\\
u_n^p(0,\cdot) &= 0,
\end{align*}
and 
\begin{align*}
\partial_t v_n^p - \Delta(\mu_n v_n^p) &= \mathds{1}_{|P_n|>p} P_n,\\
v_n^p(0,\cdot) &= 0.
\end{align*}
Thanks to Proposition \ref{propo:comp}, for any fixed $p\in\N$, $(u_n^p)_n\ddin\textnormal{L}^2(Q_T)$ (because the r.h.s. is uniformly bounded by $p$). On the other hand, we have the following estimate for $v_n^p$: using $1$ as test function in the definition of dual solution (one checks easily that $\Phi_{\mu_n}^1 = t-T$) we get 
\begin{align*}
-\langle v_n^p,1\rangle =  \langle t-T, \mathds{1}_{|P_n|>p} P_n\rangle.
\end{align*}
But since we are dealing with dual solutions, the vanishing initial data of $v_n^p$ and the inequality $P_n\leq 0$ (because $z_n$ is assumed non-negative) imply $v_n^p \leq 0$. We thus infer from the previous equality the following estimate (which amounts to formally integrate the equation satisfied by $v_n^p$ on $[0,t]\times\T^N$, and then another time on $[0,T]$)
\begin{align*}
\|v_n^p\|_{\textnormal{L}^1(Q_T)} \leq T \| \mathds{1}_{|P_n|>p} P_n\|_{\textnormal{L}^1(Q_T)}.
\end{align*}
Thanks to \eqref{unifint}, the previous inequality leads to 
\begin{align}
\label{unifconv}\sup_{n\in \N} \|v_n^p\|_{\textnormal{L}^1(Q_T)} \operatorname*{\longrightarrow}_{p\rightarrow +\infty} 0.
\end{align} 
Now, since $(u_n^p)_n\ddin\textnormal{L}^2(Q_T)$ for any fixed $p\in\N$, up to a diagonal extraction that we omit, we can thus assume that for all $p\in\N$, $(u_n^p)_n$ converges in $\textnormal{L}^2(Q_T)\hookrightarrow\textnormal{L}^1(Q_T)$. Writing, for any $m,n,p\in\N$ 
\begin{align*}
\|z_n^--z_m^-\|_{\textnormal{L}^1(Q_T)} \leq \|v_n^p\|_{\textnormal{L}^1(Q_T)} + \|u_n^p-u_m^p\|_{\textnormal{L}^1(Q_T)} + \|v_m^p\|_{\textnormal{L}^1(Q_T)},
\end{align*}
an using the uniform convergence \eqref{unifconv}, we get that $(z_n^-)_n$ is a Cauchy sequence in $\textnormal{L}^1(Q_T)$, so that it converges in this space. $\qedhere$
\end{proof}
\subsection{From non-local to classical cross-diffusion systems}\label{subsec:conv}
\begin{proof}[Proof of Theorem~\ref{thm:nonloctoclas}]
For convenience, we will sometimes write $u_{i,n}$ instead of $u_i^n$ (to avoid  overloading the exponents) and we introduce the following notations for $n\in\N$ and $i\in\{1,\dots,I\}$: $\tilde{u}_{i,n} =u_{i,n}\star\rho_i^n$, $U_n=(u_{i,n})_{1\leq i\leq I}$, $\widetilde{U}_n=(\tilde{u}_{i,n})_{1\leq i\leq I}$,  and (for $i<I$) $\tilde{a}_{i,n} = a_i(\,\cdot\,,\tilde{u}_{i+1,n},\dots,\tilde{u}_{n,I})$  so that the relaxed system in the statement of Theorem~\ref{thm:nonloctoclas} writes simply
\begin{empheq}[left = \empheqlbrace]{align*}
&\partial_t u_{1,n} - \Delta[\tilde{a}_{1,n} u_{1,n}] = r_1(\widetilde{U}_n)u_{1,n},\\
&\partial_t u_{2,n} - \Delta[\tilde{a}_{2,n} u_{2,n}] = r_2(\widetilde{U}_n)u_{2,n}, \\
&\hspace{2mm}\vdots\\
&\partial_t u_{I,n} - \Delta[a_I u_{I,n}] = r_I(\widetilde{U}_n)u_{I,n}.
\end{empheq}
The first thing to notice is that each function $u_{i,n}$ is a dual solution (recall Definition~\ref{defi:dualsol}) of the corresponding $i$th equation that it solves. Indeed, we have  $U_n\in\textnormal{L}^2(Q_T)$ so that $\widetilde{U}_n\in\textnormal{L}^2(Q_T)$ and since $r_i$ is sub-affine for all $i$, we have also $r_i(\widetilde{U}_n)\in\textnormal{L}^2(Q_T)$, using that $\tilde{a}_{i,n}$ is bounded we eventually infer $\partial_t U_n \in\textnormal{L}^1(0,T;\H^{-m}(\T^N)$ for some large integer $m\in\N$. The regularity of the kernels $\rho_i^n$ thus gives $\widetilde{U}_n\in\mathscr{C}^0([0,T];\mathscr{C}^\infty(\T^N))$ and in particular $\widetilde{U}_n\in\textnormal{L}^\infty(Q_T)$ (without any uniform bound, though). Since $r_i$ is continuous, we have also $r_i(U_n)\in\textnormal{L}^\infty(Q_T)$, using again that $U_n\in\textnormal{L}^2(Q_T)$ we recover that all the right-hand sides belong to $\textnormal{L}^2(Q_T)$. 

\vspace{2mm} 

Now that we obtained that $u_{i,n}$ is a dual solution, and since the functions $r_i$ are all upper-bounded, we infer from Proposition \ref{propo:sublin} the existence of a positive constant $C_T$ such that, for all $i$ (recall that $u_{i,n}$ is non-negative by assumption), 
\begin{align}
\label{ineq:uinvin}0 \leq u_{i,n} \leq C_T v_{i,n},
\end{align}
 where for each $i\in\{1,\dots,I-1\}$ the function $v_{i,n}$ is the dual solution of 
\begin{align}
\label{eq:vin}\partial_t v_{i,n} - \Delta[\tilde{a}_{i,n} v_{i,n}] &= 0,\\
\label{eq:vin0}v_i^n(0,\cdot) &= u_i^0
\end{align}
and likewise $v_{I,n}$ is the dual solution of 
\begin{align*}
\partial_t v_{I,n} - \Delta[a_I\, v_{I,n}] &= 0,\\
v_{I,n}(0,\cdot) &= u_I^0.
\end{align*}
Since all the functions $a_i$ are assumed bounded, we directly infer from estimate \eqref{ineq:z} of Theorem~\ref{thm:duaz} that for all $i\in \{1,\dots,I-1\}$, $(v_{i,n})\din\textnormal{L}^2(Q_T)$, and thus $(u_{i,n})_n \din\textnormal{L}^2(Q_T)$. 

\vspace{2mm}

Now, we are going to propagate strong convergence from the last (less coupled) equation up to the first one using at each step Lemma \ref{lem:comprz}. We start with the case $i=I$, that is
\begin{align*}
\partial_t u_{I,n} - \Delta[a_I u_{I,n}] = r_I(U_n)u_{I,n}.
\end{align*}
We already proved that $(U_n)_n\din\textnormal{L}^2(Q_T)$. Since $r_I$ is sub-affine, we have also $(r_I(U_n))_n\din\textnormal{L}^2(Q_T)$. Since $a_I\in\textnormal{L}^\infty(Q_T)$ does not depend on $n$, we have obviously $a_I\ddin\textnormal{L}^1(Q_T)$ and Lemma \ref{lem:comprz} applies in an easy way, because $u_{I,n}$ is non-negative and belongs to $\textnormal{L}^2(Q_T)$. We therefore have that $(u_{I,n})_n\ddin\textnormal{L}^2(Q_T)$, which is the starting point of our backward induction: we assume now that $(u_{i,n})_n\ddin\textnormal{L}^2(Q_T)$ for $i\in\{j+1,\dots,I\}$ and we will prove that $(u_{j,n})_n\ddin\textnormal{L}^2(Q_T)$. W.l.o.g. we can thus assume that $(u_{i,n})_n$ converges to $u_i$ in $\textnormal{L}^2(Q_T)$ for $i\in\{j+1,\dots,I\}$. Since $(\rho_i^n)_n$ approaches the Dirac mass, it is an easy exercise to check that for $i\in\{j+1,\dots,I\}$ we have also that $(u_{i,n}\star\rho_i^n)_n$ converges in $\textnormal{L}^2(Q_T)$ to $u_i$ (for instance: the previous sequence is $\textnormal{L}^2(Q_T)$ compact by the Riesz-Fréchet-Kolmogorov criterion, and only one limit point is possible thanks to Dirac mass approximation). Since $a_j$ is continuous and depends only on the $u_{i,n}\star\rho_i^n$ for $i\geq j+1$, we have that (up to a subsequence) $\tilde{a}_j^n$ converges almost everywhere to $a_j(\,\cdot\,,u_{j+1},\dots,u_I)$ and since $\tilde{a}_j^n$ is bounded by $\|a_j\|_\infty$, the convergence actually holds in $\textnormal{L}^1(Q_T)$ for $\tilde{a}_j^n$. As before, since $r_j$ is sub-affine, we have $(r_j(U_n))_n\din\textnormal{L}^2(Q_T)$ and we invoke Lemma \ref{lem:comprz} to get $(u_{j,n})_n \ddin\textnormal{L}^2(Q_T)$.  

\vspace{2mm} 

To conclude we now have to exploit the previous convergences to pass to the limit in each nonlinear terms of the relaxed system. Thanks to the continuity of the functions $a_i$ and $r_i$, we have a.e. convergence in each nonlinearities, so only concentration issues can occur. Since the functions $a_i$ are bounded, the nonlinearities lying inside the laplacian are harmless. As for the reaction terms, we established $(r_i(U_n))_n\din\textnormal{L}^2(Q_T)$ so that up to a subsequence we have weak convergence in $\textnormal{L}^2(Q_T)$ towards a limit which is identified thanks to the a.e. convergence of the $(u_{i,n})_n$ towards $u_i$ and the continuity of $r_i$: $(r_i(U_n))_n\rightharpoonup r_i(U)$, where $U=(u_i)_{1\leq i\leq I}$. This weak convergence is facing the strong one of $(u_i^n)_n$, so that indeed $(r_i(U_n)u_i)_n\rightharpoonup r_i(U)u_i$ for all $i\in\{1,\dots,I\}$. $\qedhere$
\end{proof}
\subsection{An existence result for non-local triangular systems}\label{subsec:exi}
\begin{proof}[Proof of Theorem~\ref{thm:exi}]
When $a_i,r_i$ and $u^{in}_i$ are smooth, with $a_i$ bounded from above and below the existence of a unique non-negative (smooth) solution to this triangular system is a consequence of Theorem 3.1 and Lemma 6.3 of \cite{melfont}. One could also follow a PDE approach as this is done in \cite{LPR}, after adapting the analysis to the convolution operator and adding the reaction terms (which are harmelss for their method). The hard part of the job is of course to handle the low regularity case. We proceed by approximation and replace $a_i$, $r_i$ and $u^{in}_i$ by sequences of smooth functions $(a_{i,n})_n$, $(r_{i,n})_n$ and $(u^{in}_n)_n$ approaching them, locally uniformly for the two first ones, in $\textnormal{L}^2(\T^N)$ for the last one. We also demand that $(a_{i,n})_n$ and $(r_{i,n})_n$ satisfy, uniformly in $n$, the pointwise assumptions that we have for $a_i$ and $r_i$ : upper and lower bounds, sub-affineness. If $(u_{i,n})_{1\leq i\leq I}$ is the corresponding family solution of the regularized non-local triangular system, since $u_{i,n}$ is smooth, it is in particular a dual solution of the corresponding $i$th equation that it solves (recall Definition~\ref{defi:dualsol}). The very same reasoning that we used in the proof of Theorem~\ref{thm:nonloctoclas} applies, even in a simpler way : Proposition \ref{propo:comp} is sufficient to recover strong compactness. Indeed, since $(u_{i,n})_n$ is bounded in $\textnormal{L}^2(Q_T)$, the equation that it solves implies that $(\partial_t u_{i,n})_n$ is bounded in $\textnormal{L}^1(0,T;\H^{-m}(\T^N))$ for some large integer $m\in\N$ (this part is identical to the proof of Theorem~\ref{thm:nonloctoclas}). The regularity of the kernels $\rho_i$ (which here are fixed independenty of $n$) ensures thus that for all $i$, $(u_{i,n}\star\rho_i)_n \din\textnormal{L}^\infty(Q_T)$. Since $r_{i,n}$ approaches the continuous function $r_i$ locally uniformly, we have in particular that $(r_{i,n}(u_{1,n}\star\rho_1,\cdots,u_{I,n}\star\rho_I))_n\din\textnormal{L}^\infty(Q_T)$ and one can then use directly Proposition \ref{propo:comp} to get $(u_{i,n})_n\ddin\textnormal{L}^2(Q_T)$ for all $i$ and conclude as we did for Theorem~\ref{thm:nonloctoclas}. $\qedhere$
\end{proof}
\section{Comments}\label{sec:com}
\subsection{A (seemingly) simple open problem for the dual equation}\label{subsec:open}
We go back here to Proposition~\ref{propo:duaphi}. As only the $\textnormal{L}^1(Q_T)$ norm of $\mu$ is used in estimate \eqref{ineq:apriori2}, it is natural to ask whether this proposition still holds when $\mu$ is positively lower-bounded and (only) integrable. The existence part is not problematic.  What we are interested herein is how much one can weaken the assumption on $\mu$ so as to keep the uniqueness for $\Phi$, because the latter property is at the core of all the results of the current paper. When one tries to reproduce directly the proof in that case, the natural constraint that appears for the source term $S$ is $\mu^{-1/2}S\in\textnormal{L}^2(Q_T)$ and we get then the following functional framework for $\Phi$: $\textnormal{L}^\infty(0,T;\H^1(\T^N))\cap\textnormal{L}^2(0,T;\H^2(\T^N))\cap\mathscr{C}^0([0,T];\textnormal{L}^2(\T^N))$ with furthermore $\mu^{1/2}\Delta\Phi \in\textnormal{L}^2(Q_T)$ and $\mu^{-1/2}\partial_t \Phi\in\textnormal{L}^2(Q_T)$. Note that the two last belongings are not equivalent to $\Delta\Phi\in\textnormal{L}^2(Q_T)$ and $\partial_t \Phi\in\textnormal{L}^2(Q_T)$, when $\mu$ is only bounded from below. The proof of the \emph{a priori} estimate \eqref{ineq:apriori1} in Proposition~\ref{propo:duaphi} becomes then problematic. The issue comes in when one tries to establish the following inequality (which was true for the class of solutions considered in Proposition~\ref{propo:duaphi}):
\begin{align}
\label{ineq:prob} \frac{1}{2}\|\nabla \Phi\|_{\textnormal{L}^\infty(0,T;\textnormal{L}^2(\T^N))}^2 \leq \int_{t}^T \int_{\T^N} \partial_t \Phi\Delta \Phi\,.
\end{align}
To prove it, one should replace $\Phi$ by an appropriate approximation, a sequence of regular functions $(\Phi_k)_k$ such that in particular $(\mu^{-1/2}\partial_t \Phi_k,\mu^{1/2}\Delta\Phi_k)$ approaches $(\mu^{-1/2}\partial_t \Phi,\mu^{1/2}\Delta\Phi)$ in $\textnormal{L}^2(Q_T)\times\textnormal{L}^2(Q_T)$. The existence of such a sequence is related to the approximation problem ``$H=W$'' for weighted Sobolev spaces, whence the interrogation : \emph{are smooth functions always dense in weighted Sobolev spaces} ? This is an intricate question which unfortunately cannot always be answered positively (see \cite{zhikov1998} for a simple counterexample and the more recent \cite{zhikov2013} for refined sufficient conditions). This issue is somehow related to the method we used to prove Theorem~\ref{thm:duaz}, which needed the density of $\big\{ \mu^{-1/2}\partial_t \ffi +\mu^{1/2}\Delta \ffi\,:\,\ffi\in\mathscr{D}(Q_T)\big\}$ in $\textnormal{L}^2(Q_T)$. Without uniqueness for the dual problem it is not clear at all that the formulation \eqref{eq:form:strong} is equivalent to  \eqref{eq:def:dist}. Conversely, with such an approximation property at hand for the associated weighted Sobolev spaces, the aforementioned density is clear. 

\medskip 

To sum up, after several attempts we did not manage to answer to the following question.

\begin{question}\label{ques}
Fix $\mu\in\textnormal{L}^1(Q_T)$ positively lower-bounded and consider $\Phi\in\textnormal{L}^\infty(0,T;\H^1(\T^N))\cap\textnormal{L}^2(0,T;\H^2(\T^N))\cap\mathscr{C}^0([0,T];\textnormal{L}^2(\T^N))$ such that $\mu^{-1/2}\partial_t \Phi\in\textnormal{L}^2(Q_T)$, $\mu^{1/2}\Delta \Phi\in\textnormal{L}^2(Q_T)$ and solving
\begin{align*}
\Phi(T,\cdot)&=0,\\
\partial_t\Phi+\mu\Delta\Phi&=0.
\end{align*}
 Do we have $\Phi=0$ ?
\end{question}
As explained above, our analysis rests on the well-posedness for the dual problem. The more we extend the class of diffusion coefficients for which the answer to the previous question is positive, the more we recover cases for which our compactness routine will apply (see next paragraph). Here are some cases in which the previous uniqueness result holds.
\begin{lem}\label{lem:ques}
In the four following cases the answer to Question \ref{ques} is positive:
\begin{itemize}
\item[(i)] $\mu$ is bounded ;
\item[(ii)] $\mu$ depends only on $x$ ;
\item[(iii)] $\mu$ depends only on $t$ ;
\item[(iv)] $\mu$ satisfies the $A_2$-Muckenhoupt condition ;
\end{itemize}
\end{lem}
\begin{remark}
For the definition of the $A_2$-Muckenhoupt condition, see the Appendix Section \ref{sec:app}.
\end{remark}
\begin{proof}
Case $(i)$ was known at least since \cite{campanato}. In the current manuscript, it is treated in Proposition \ref{propo:duaphi} (and is in fact a particular instance of case $(iv)$). For case $(ii)$, we have $\Psi:=\mu^{-1/2}\Phi\in \textnormal{W}^{1,2}(0,T;\textnormal{L}^2(\T^N))$ which is sufficient to prove that (we use $\Psi(T,\cdot)=0$) 
\begin{align*}
 \int_{Q_T} \Psi \,\partial_t\Psi = -\frac{1}{2}\|\Psi(0)\|_{\textnormal{L}^2(\T^N)}^2.
\end{align*}
In particular multiplying the equation satisfied by $\Phi$ by $-\mu^{-1}\Phi$, we get (since $\Phi\in\textnormal{L}^2(0,T;\H^2(\T^N))$) and integrating over $Q_T$
\begin{align*}
\frac{1}{2}\|\Psi(0)\|_{\textnormal{L}^2(\T^N)}^2 + \|\nabla \Phi\|_{\textnormal{L}^2(Q_T)}^2 = 0,
\end{align*}
from which we easily get $\Phi=0$. For case $(iii)$ we introduce a mollifier sequence $(\eta_\ep)_\ep$ with respect to the space variable and define (convolution with respect to $x$ only) $\Psi_\ep := \Phi\star\eta_\ep$. Since $\partial_t\Phi\in\textnormal{L}^1(Q_T)$ and $\Phi\in\mathscr{C}^0([0,T];\textnormal{L}^2(\T^N))$, we infer $\partial_t \Psi_\ep \in\textnormal{L}^1(0,T;\mathscr{C}^\infty(\T^N))$ and $\Psi_\ep \in\mathscr{C}^0([0,T];\mathscr{C}^\infty(\T^N))$ ; again this is sufficient to write (note that $\Psi_\ep(0,\cdot)=0$ because the convolution is in space only) 
\begin{align}
\label{eq:psiep}\int_{Q_T} \Delta \Psi_\ep\, \partial_t\Psi_\ep = \frac{1}{2}\|\nabla \Psi_\ep(0)\|_{\textnormal{L}^2(\T^N)}^2.
\end{align}
Since $\Psi_\ep$ solves the same equation as $\Phi$ (because $\mu$ does not depend on $x$), multiplying it by $\Delta \Psi_\ep$ and integrating over $Q_T$ we get 
\begin{align*}
\frac{1}{2}\|\Psi_\ep(0)\|_{\textnormal{L}^2(\T^N)}^2 + \|\mu^{1/2}\Delta \Psi_\ep\|_{\textnormal{L}^2(Q_T)}^2 = 0,
\end{align*}
which again leads easily to $\Psi_\ep=0$. Since $(\Psi_\ep)_\ep$ converges to $\Phi$, we infer $\Phi=0$.

\vspace{2mm}

For the last case $(iv)$ recall that the $A_2$-Muckenhoupt condition characterizes those weights $\nu$ such that the  maximal operator $M$ is continuous from $\textnormal{L}^2(Q_T;\nu\,\dd x\,\dd t)$ to itself (see the Appendix Section \ref{sec:app}). As we did in case $(iii)$ we introduce $\Psi_\ep:=\Phi\star\eta_\ep$ (convolution in $x$ only) which again satisfies \eqref{eq:psiep}. But $\Psi_\ep$ does not satisfy the same equation as $\Phi$. We introduce the weighted space $\textnormal{L}^2_\mu(Q_T):=\textnormal{L}^2(Q_T;\mu\, \dd x \,\dd t)$ and similarly $\textnormal{L}^2_{1/\mu}(Q_T)$ and define
\begin{align*}
S_\ep := \partial_t \Psi_\ep + \mu \Delta \Psi_\ep.
\end{align*}
We claim that $(S_\ep)_\ep$ converges to $0$ in $\textnormal{L}^2_{1/\mu}(Q_T)$. We have more precisely that $(\partial_t \Psi_\ep)_\ep$ and $(\mu\,\Delta\Psi_\ep)_\ep$ converge respectively to $\partial_t \Phi$ and $\mu\,\Delta\Phi$ in $\textnormal{L}^2_{1/\mu}(Q_T)$. Note that the second convergence is equivalent to $(\Delta\Psi_\ep)_\ep\rightarrow\Delta\Phi$ in $\textnormal{L}^2_\mu(Q_T)$. First, by standard properties of convolution we have that $(\partial_t\Psi_\ep,\Delta\Psi_\ep)_\ep\rightarrow (\partial_t\Psi,\Delta\Psi)$ almost everywhere. The maximal function allows to control ponctually an approximate convolution by means of the function itself (see the Appendix Section \ref{sec:app}), more precisely up to a universal constant we have for all $\ep>0$, $|\Psi_\ep|\lesssim M|\Phi|$. And since $\partial_t\Phi\in\textnormal{L}^1(Q_T)$ and $\Delta\Phi\in\textnormal{L}^2(Q_T)$, we have also (with the same constant), $|\partial_t \Psi_\ep|\lesssim M|\partial_t \Phi|$ and $|\Delta\Psi_\ep|\lesssim  M|\Delta\Phi|$. Since the maximal operator is continuous from $\textnormal{L}^2_{1/\mu}(Q_T)$ to itself and $\textnormal{L}^2_{\mu}(Q_T)$ to itself, we have that $\mu^{-1/2}M|\partial_t\Phi|\in\textnormal{L}^2(Q_T)$ and $\mu^{1/2}M|\Delta\Phi|\in\textnormal{L}^2(Q_T)$, in particular these two functions can be used to apply Lebesgue's dominated convergence theorem to prove that $(\partial_t \Psi_\ep)_\ep\rightarrow \partial_t\Phi$ and $(\Delta \Psi_\ep)_\ep\rightarrow \Delta\Phi$ respectively in $\textnormal{L}^2_{1/\mu}(Q_T)$ and $\textnormal{L}^2_\mu(Q_T)$. This proves in particular that $(S_\ep)_\ep$ converges to $0$ in $\textnormal{L}^2_{1/\mu}(Q_T)$, and since $(\Delta\Psi_\ep)_\ep$ is bounded (because it converges) in $\textnormal{L}^2_\mu(Q_T)$, we have that $(\Delta \Psi_\ep \,S_\ep)_\ep\rightarrow 0$ in $\textnormal{L}^1(Q_T)$. Multiplying 
\begin{align}
\label{eq:ep}\partial_t\Psi_\ep + \mu\Delta\Psi_\ep = S_\ep,
\end{align}
by $\Delta\Psi_\ep$ and integrating over $Q_T$ we get, thanks to \eqref{eq:psiep}
\begin{align*}
\frac{1}{2}\|\nabla \Psi_\ep(0)\|_{\textnormal{L}^2(\T^N)}^2 + \|\Delta\Psi_\ep\|_{\textnormal{L}^2_\mu(Q_T)}^2 = \int_{Q_T} \Delta \Psi_\ep\, S_\ep \operatorname*{\longrightarrow}_{\ep\rightarrow 0} 0,
\end{align*}
so we get in particular that $(\Delta \Psi_\ep)_\ep$ converges to $0$ in $\textnormal{L}^2_\mu(Q_T)$, and because of \eqref{eq:ep} we have the same behavior for $(\partial_t \Psi_\ep)_\ep$ in $\textnormal{L}^2_{1/\mu}(Q_T)$ : we get $\partial_t \Phi= \Delta\Phi=0$, so $\Phi=0$. $\qedhere$
\end{proof}
\subsection{Perspectives}
Using Lemma \ref{lem:ques} it is possible to strengthen the compactness result Lemma \ref{lem:comprz}. All the properties of dual solutions actually hold as soon as the diffusion functions $\mu$ or $\mu_n$ belong to an ``admissible class'' for which the dual problem have a unique solution (as examples are given in Lemma \ref{lem:ques}) ; the constraint of boundedness for the diffusion coefficient is completely artificial and is only here to ensure that the dual equation is well-posed. As implied by Remark \ref{rem:obs}, one could generalize Theorem~\ref{thm:nonloctoclas} replacing the $a_i$ functions by (global) integral operators of the time or space variable, and invoke points $(ii)$ or $(iii)$ of Lemma \ref{lem:ques} to recover well-posedness of the dual problem. Points $(iv)$ of Lemma \ref{lem:ques} could be of interest if one manages to prove additional estimates on solution to triangular SKT systems, which would imply the $A_2$-Muckenhoupt condition ; however this trail is a bit speculative for the moment. The Grail would be the generalization of Lemma \ref{lem:comprz} to the case of integrable diffusion coefficients $\mu_n$ and $\mu$, which amounts to answer to Question \ref{ques} positively, a task in which we did not manage to succeed. The situation is a bit frustrating: the current form of Lemma \ref{lem:comprz} does not exploit the full power of the duality estimate which is very useful to avoid concentration issues, especially for the terms inside the laplacian operator ; with a bounded diffusion coefficient this feature is unfortunately not used and in fact, the concentration in the reaction terms can be direclty avoided in this case using the improved duality estimates obtained in \cite{dualimpro}, see for instance \cite{DesTres,ariane}, where this strategy is used. In its current form our compactness result is mainly useful to control oscillations, allowing the transference of a.e. convergence from the diffusion coefficient to the solution of the equation itself, without the help of parabolic estimates (which necessitates regularity of the coefficient and the inital data).
\section{Appendix}\label{sec:app}
\begin{definition}
A positive function $\nu\in\textnormal{L}^1(Q_T)$ is said to satisfy the $A_2$-Muckenhoupt condition if 
\begin{align*}
\sup_B \left(\fint_B \nu\right)\left(\fint_B \frac{1}{\nu}\right) <\infty,
\end{align*} 
where the supremum runs on all balls $B\subset\T^N$ and for $f\in\textnormal{L}^1(B)$ 
\begin{align*}\fint_B f = \frac{1}{|B|}\int_B f.\end{align*}
\end{definition}
\begin{definition}
For any $f\in\textnormal{L}^1(\T^N)$ the associated maximal function, denoted $Mf$, is defined by
\begin{align*}
Mf(x) := \sup_{r>0}\fint_{B_r(x)} |f|,
\end{align*}
where $B_r(x)$ is the ball of radius $r$ centered at $x$. The maximal operator is then the linear map defined by $f\mapsto M f$.
\end{definition}
\begin{lem}
Fix an approximate identity $(\rho_\ep)_\ep$ on $\T^N$. There exists $A>0$ such that, for any $p\in[1,\infty]$ and any $f\in\textnormal{L}^p(\T^N)$, $|f\star\rho_\ep| \leq A M|f|$ almost everywhere.
\end{lem}
\begin{proof}
See \cite{stein}, Chapter III, Section 2, Theroem 1, p.63. $\qedhere$
\end{proof}
\begin{thm}[Muckenhoupt]
  Consider $\nu\in\textnormal{L}^1(Q_T)$ and let $\textnormal{L}^2_\nu(Q_T):=\textnormal{L}^2(Q_T;\nu\,\dd x\,\dd t)$. Then the maximal operator $M$ is continuous from $\textnormal{L}_\nu^2(Q_T)$ to itself if and only if $\nu$ satisfies the $A_2$-Muckenhoupt condition.
\end{thm}
\begin{proof}
See \cite{muck}.$\qedhere$
\end{proof}
\vspace{2mm}
{\bf{Acknowledgement}}: 
The author would like to thank Thomas Lepoutre for several fruitful discussions. The research leading to this paper was funded by the french "ANR blanche" project Kibord: ANR-13-BS01-0004.
\bibliographystyle{plain}
\bibliography{nonloc2crosstribound} 
\end{document}